\documentclass[a4paper,11pt]{amsart}

\usepackage[margin=0.95in]{geometry}

\usepackage{amsmath}
\usepackage{amssymb}
\usepackage{xcolor}
\usepackage{bbm}
\usepackage{graphicx}
\usepackage{verbatim}
\usepackage{algorithm}

\usepackage[T1]{fontenc}
\usepackage[ansinew]{inputenc}
\usepackage[german, french,USenglish]{babel}

\makeatletter

\theoremstyle{plain}

\newtheorem{theorem}{Theorem}

\newtheorem{proposition}[theorem]{Proposition}

\newtheorem{corollary}[theorem]{Corollary}

\theoremstyle{remark}

\newcommand{\R}{\mathbb{R}}

\renewcommand{\epsilon}{\varepsilon}

\newtheorem{assumption}{Assumption}

\def\Aset{\mathsf{A}}

\usepackage[colorinlistoftodos,bordercolor=orange,backgroundcolor=orange!20,linecolor=orange,textsize=scriptsize]{todonotes}

\numberwithin{equation}{section}

\def\JELname{{\bfseries JEL Classification}\enspace}
      \def\JEL#1{\par\addvspace\medskipamount{\rightskip=0pt plus1cm
      \def\and{\ifhmode\unskip\nobreak\fi\ $\cdot$
      }\noindent\JELname\ignorespaces#1\par}}

\makeatletter
\newcommand{\leqnomode}{\tagsleft@true\let\veqno\@@leqno}
\newcommand{\reqnomode}{\tagsleft@false\let\veqno\@@eqno}
\makeatother

\begin{document}
\title[Mesh algorithms  for MDP]
{
Weighted mesh algorithms for general Markov decision processes: Convergence and tractability }
\author[D.~Belomestny]{Denis Belomestny$^{1}$}
\address{$^1$Faculty of Mathematics\\
Duisburg-Essen University\\
Thea-Leymann-Str.~9\\
D-45127 Essen\\
Germany}
\email{denis.belomestny@uni-due.de, veronika.zorina@uni-due.de}

\author[J.~Schoenmakers]{John Schoenmakers$^{2}$}
\address{$^2$Weierstrass Institute for Applied Analysis and Stochastics \\
\mbox{Mohrenstr.~39} \\
10117 Berlin \\
Germany}
\email{schoenma@wias-berlin.de}

\author[V.~Zorina]{ Veronika Zorina$^{1}$}

\keywords{Markov decision processes, mesh method, tractability}
\subjclass[2010]{90C40\and 65C05\and 62G08}

\date{}
\maketitle
\begin{abstract}
We introduce a mesh-type approach for tackling discrete-time, finite-horizon Markov Decision Processes (MDPs) characterized by state and action spaces that are general, encompassing both finite and infinite (yet suitably regular) subsets of Euclidean space. In particular, 
for bounded state and action spaces,  
our algorithm achieves a computational complexity that is tractable in the sense 
of Novak \& Wo\'{z}niakowski \cite{NovakWozniakowski2008}, and is polynomial in the time horizon.
For unbounded state space the algorithm 
is ``semi-tractable'' in the sense that the
complexity is proportional to 
$\epsilon^{-c}$ 
with some dimension independent $c\geq2$,
for 
achieving an accuracy $\epsilon$, and  polynomial 
in the time horizon with degree linear in the underlying dimension. As such the proposed approach has some flavor 
of the randomization method by Rust~\cite{rust1997using} which deals with infinite horizon MDPs and uniform sampling in compact state space. However, the present approach is essentially different due to the finite horizon and
a simulation procedure due to general transition distributions, and more general in the sense that it encompasses  unbounded state space. 
To demonstrate the effectiveness of our algorithm, we provide illustrations based on Linear-Quadratic Gaussian (LQG) control problems.
\end{abstract}

\section{Introduction}

Markov decision processes (MDPs) provide a general framework for modeling sequential decision-making under uncertainty. A large number of practical problems from various areas such as economics, finance, and machine learning can be viewed as MDPs.
{For a classical reference
we refer to \cite{puterman2014markov}, and for MDPs with application to finance, see
\cite{BaRi}.}
The aim is usually to find an optimal policy that maximizes the expected accumulated rewards (or minimizes the expected accumulated costs).
{In principle, these Markov decision  problems can be  solved by an approximate dynamic programming approach, see \cite{puterman2014markov};}
 however, in practice, this approach suffers from the so-called ``curse of dimensionality'' and the ``curse of horizon''  meaning that the complexity (running time) of the program increases exponentially in the dimension of the problem (dimensions of the state and action spaces) and the horizon (or effective horizon for discounted infinite horizon MDPs).  Traditional dynamic programming (DP) algorithms, such as value- or policy-iteration, exhibit exponential scaling with MDP size, even when coupled with advanced multigrid algorithms, see \cite{BelSch2023MDP}  for recent review of various approximative DP algorithms for general state and action spaces.
 Furthermore, the curse of dimensionality can be considered   as a lower bound on the complexity of any MDP,
 not confined to any specific algorithm, as evidenced by Chow and Tsitsiklis \cite{chow1989complexity}.
 \par
The problem of solving MDPs without curse of dimensionality  attracted a lot of attention in the literature.  The first work in this direction was Rust~\cite{rust1997using} where the author proposed 
in an 
infinite horizon setting a weighted mesh  algorithm 
with complexity proportional to $\epsilon^{-4}$ for a target accuracy $\epsilon$ and a polynomial in the underlying dimension. 
Another approach based on Monte Carlo tree search and sparse sampling  was suggested in Kearns et al.~\cite{kearns2002sparse}. In particular, the authors in \cite{kearns2002sparse} demonstrated that a specific online tree-building algorithm successfully circumvents the curse of dimensionality in discounted MDPs. This achievement has been further extended to partially observable MDPs (POMDPs) by the same authors in \cite{kearns1999approximate}. The bounds established in these two papers remain independent of the dimension of the state space but exhibit exponential scaling with $1/(1-\gamma)$, representing the effective horizon-time, where $\gamma$ is the discount factor of the MDP. Moreover, the complexity  depends on the number of actions polynomially with power again proportional to the effective horizon-time.    A recent work \cite{beck2023nonlinear} proposed a nonlinear Multilevel Monte Carlo approach to solve infinite horizon MDPs without curse of dimensionality. Note that  the complexity estimates in \cite{beck2023nonlinear} is of order $\epsilon^{-c}$ with $c$  depending on the effective horizon-time. Moreover, the  number of actions is assumed to be finite. 
Let us stress that the setting of finite horizon MDPs is essentially different from the infinite horizon one where  we need to solve a fixed  point problem. Finite horizon setting requires a backward dynamic programming procedure and  simulation of the paths of the underlying Markov process instead of the one step transitions as in the infinite horizon MDPs. As a result, the convergence analysis of the algorithms in finite horizon MDPs becomes much more intricate.

\par
In this paper, we present a novel approach for addressing high-dimensional finite horizon Markov Decision Processes (MDPs) using a weighted mesh approach. This methodology shares conceptual similarities with the approach proposed by Rust~\cite{rust1997using} (see also \cite{szepesvari2001efficient}). However, it's essential to note that the work by Rust focuses on infinite horizon discounted MDPs, introducing a crucial distinction between the two settings.
Unlike Rust, who can independently sample at each step of the iteration procedure, our approach involves drawing trajectories of the underlying state process and proceeding backwardly. This results in a more intricate structure of weights and their dependence on controls. Additionally, Rust's work imposes rather restrictive assumptions on the underlying MDP, assuming, for instance, a compact state space, finite action space 
and  transition densities uniformly bounded away from zero. These assumptions exclude consideration of many interesting cases, such as Gaussian  processes with non-compact supports (refer to \cite{bray2022comment} for a discussion on this and related issues). In our work, we allow for non-compact state spaces, continuous action spaces and general classes of transition densities. Also note that the proposed methodology essentially differs from the one adopted in \cite{dufour2015approximation} for infinite horizon MDPs, where the authors use simple empirical means with respect to a control-independent reference  measure and derive bounds conditional on a ``good'' event where the denominator of the estimator is bounded away from zero. In contrast, we derive convergence rates in the \( L_1 \) norm with explicit dependence on the complexity of the state and action spaces.
\par
Thus, paper's primary contribution is the introduction of a new weighted mesh algorithm designed for a broad range of finite horizon Markov Decision Processes (MDPs). We have also conducted a thorough  complexity analysis of this algorithm. Our findings reveal that this algorithm is capable of efficiently solving a wide spectrum of finite horizon MDPs, including those with non-compact state/action spaces that are subsets of \(\mathbb{R}^d\) and feature general transition densities. Significantly, the computational complexity of our algorithm, denoted as \(\mathcal{C}(\epsilon, d)\), demonstrates a polynomial dependence on the horizon length, ensuring \(\epsilon\)-accuracy in approximating the corresponding value functions at a given point. Moreover, it holds  
\[
\lim_{d\rightarrow\infty}\lim_{\epsilon\searrow0}\frac{\log\mathcal{C}%
\left(  \epsilon,d\right)  }{f(d)\log(1/\epsilon)}=0
\]
for any $f$ with arbitrary slow convergence to infinity  as $d\to\infty$.  This type of dependency on $d$ and $\epsilon$ can be characterized as ``semi-tractable'' or indicative of a ``weak curse of dimensionality.'' To our knowledge, this marks the first instance in the  literature where a general finite horizon Markov Decision Process (MDP) is approximated with an algorithm that exhibits at most a ``semi-tractable'' level of complexity. 
\par
The paper is organized as follows. The basic setup of the Markov Decision Process and the well-known representations
for its maximal expected reward is given in Section~\ref{setup}.
Appendix~A introduces some auxiliary notions
needed to formulate an auxiliary result in Appendix~B
stemming from the theory of empirical processes.

\section{Setup and basic properties of the Markov Decision Process}\label{setup}

We consider the discrete time finite horizon Markov Decision Process (MDP), given by the tuple
\begin{equation*}
\mathcal{M}=(\mathsf{S},\mathsf{A},(P_{h})_{h\in]H]},(R_{h})_{h\in[H[},F,H),
\end{equation*}
made up by the following items:
\begin{itemize}
\item a measurable state space $(\mathsf{S},\mathcal{S},\rho_\mathsf{S})$%
;
\item a measurable action space $(\mathsf{A},\mathcal{A},\rho_\mathsf{A})$
;
\item an integer $H$ which defines the horizon of the problem;
\item for each $h\in ]H],$ with $]H]:=\{1,\ldots,H\}$\footnote{We further write $[H]:=\{0,1,\ldots,H\}$ etc.}, a time dependent transition function $P_{h}:$ $\mathsf{S}\times\mathsf{A}\to\mathcal{P}(\mathsf{S})$
where $\mathcal{P}(\mathsf{S})$ is the space of probability measures on
$(\mathsf{S},\mathcal{S})$;
\item a time dependent reward function $R_{h}:$ $\mathsf{S}\times\mathsf{A}\to\mathbb{R},$
where $R_{h}(x,a)$ is the immediate reward associated with taking action
$a\in\mathsf{A}$ in state $x\in\mathsf{S}$ at time step $h\in[H[$;
\item a terminal reward $F:$ $\mathsf{S}\to\mathbb{R}$.
\end{itemize}
Introduce a filtered probability space $\mathfrak{S}:=\bigl(\Omega
,\mathcal{F},(\mathcal{F}_{t})_{t\in\lbrack H]},\mathbb{P}\bigr).$ 
For a fixed policy $\boldsymbol{\pi}=(\pi_{0},\ldots,\pi_{H-1})$ with $\pi
_{t}:$ $\mathsf{S}\rightarrow\mathcal{P}(\Aset),$ we consider an adapted
controlled process $\bigl(S_{t},A_{t}\bigr)_{t=h,\ldots,H}$ on $\mathfrak{S}$
satisfying $S_{0}\in\mathsf{S},$ $A_{0}\sim\pi_{0}(S_0),$ and
\begin{eqnarray}
\label{eq:chain0}
S_{t+1}\sim P_{t+1}(\left.  \cdot\right\vert S_{t},A_{t}),\quad A_{t}\sim
\pi_{t}(S_{t}),\quad t=0,\ldots,H-1.
\end{eqnarray}
Henceforth we denote by $a_{<h}$ the deterministic vector of actions $a_{<h}%
=(a_{0},\ldots,a_{h-1})\in\mathsf{A}^{h},$ similarly $a_{\leq h}$ etc.,
\begin{assumption}\label{eq:iterfunc-repr}
In the sequel we  assume that
chain $(S_{t}(a_{<t}))$  comes from  the system of so-called \textit{random iterative
functions}:
\begin{equation*}
S_{t}=\mathcal{K}_{t}(S_{t-1},a_{t-1}, \varepsilon
_{t}),\quad t\in]H], %
\end{equation*}
where $\mathcal{K}_{t}:$ $\mathsf{S}\times\mathsf{A}\times\mathsf{E}%
\rightarrow\mathsf{S}$ is a measurable map with $\mathsf{E}$ being a
measurable space, and $(\varepsilon_{t},t\in]H])$ is an i.i.d. sequence of
$\mathsf{E}$-valued random variables defined on a probability space
$(\Omega,\mathcal{F},\mathrm{P}).$ 
\end{assumption}
Let us note that Assumption~\ref{eq:iterfunc-repr} is included for clarity, but is not a real assumption in fact. It holds for any controlled Markov model by stochastic realization arguments, see e.g. \cite[Lemma 1.2]{Bor}, \cite[Lemma 3.1]{GikSko}. 
The expected reward of this  MDP due to the chosen policy $\boldsymbol{\pi}$ is given by
\[
V_{0}^{\boldsymbol{\pi}}(x):=\mathbb{E}_{\boldsymbol{\pi},x}\left[
\sum_{t=0}^{H-1}R_{t}(S_{t},A_{t})+F(S_{H})\right], \quad x\in \mathsf{S}
\]
where \(\mathbb{E}_{\boldsymbol{\pi},x}\) stands for expectation induced by the policy \(\boldsymbol{\pi}\) and transition kernels \(P_{t},\) \(t\in [H],\)  conditional on the event \(S_0=x.\)
The
goal of the Markov decision problem is to determine the maximal expected
reward:
\begin{equation}
V_{0}^{\boldsymbol{\star}}(x_0):=\sup_{\boldsymbol{\pi}\in \Pi}\mathbb{E}%
_{\boldsymbol{\pi},x_0}\left[  \sum_{t=0}^{H-1}R_{t}(S_{t},A_{t})+F(S_{H})\right]
=\sup_{\boldsymbol{\pi}\in \Pi}V_{0}^{\boldsymbol{\pi}}(x_{0})\label{Vopt}%
\end{equation}
where \(\Pi\) is a set of all measurable mappings \((\mathsf{S}\rightarrow\mathcal{P}(\Aset))^{\otimes H}.
\)
Let us introduce for a generic time $h\in\left[H\right],$ the value
function due to the policy $\boldsymbol{\pi},$
\begin{align*}
V_{h}^{\boldsymbol{\pi}}(x)  & :=\mathbb{E}%
_{\boldsymbol{\pi},x}\left[
\left.  \sum_{t=h}^{H-1}R_{t}(S_{t},A_{t})+F(S_{H})\right\vert S_{h}=x\right]
,\text{ \ \ }x\in\mathsf{S}.
\end{align*}
Furthermore, let
\begin{equation}
\label{eq:optim_vh}
V_{h}^{\boldsymbol{\star}}(x)   :=\sup_{\boldsymbol{\pi}}%
V_{h}^{\boldsymbol{\pi}}(x)
\end{equation}
be the optimal value function at $h\in\left[  H\right]$. 
\begin{theorem}\label{Bellman}
Assume that for all $t\in[H[,$ the mappings $\mathcal{K}_{t}(\cdot,\cdot,\varepsilon)$ for any fixed $\varepsilon\in \mathsf{E},$ $R_t$ and $F$ are  uniformly  bounded and continuous functions on \(\mathsf{S}\times \mathsf{A} \) and \(\mathsf{S},\) respectively.  For  any fixed $x\in\mathsf{S},$  it then holds  $V_{H}
^{\boldsymbol{\star}}(x)=F(x),$ and
\begin{equation}
V_{h}^{\boldsymbol{\star}}(x)=\sup_{a\in A}\left(  R_{h}(x,a)+\mathbb{E}%
_{S_{h+1}\sim P_{h+1}(\cdot|x,a)}\left[  V_{h+1}^{\boldsymbol{\star}}%
(S_{h+1})\right]  \right)  ,\quad h=H-1,\ldots,0.\label{eq:bellman-q}%
\end{equation}
Moreover, the 
supremum in (\ref{eq:bellman-q}) is attained at some deterministic optimal action
$a^{\boldsymbol{\star}}=\pi_{h}^{\boldsymbol{\star}}(x)$. That is, there exists an optimal policy  solving \eqref{eq:optim_vh} which depends on $S_{t}$ in a deterministic
way. In this case, we shall write $\boldsymbol{\pi}^{\boldsymbol{\star}}%
=(\pi_{t}^{\boldsymbol{\star}}(S_{t}))$ for some mappings $\pi_{t}^{\star}:$
$\mathsf{S}\rightarrow\Aset$.
\end{theorem}
Note that Theorem~\ref{Bellman} holds under much weaker conditions, 
see e.g. \cite[Section 2.3]{BaRi}.  

Let us further introduce recursively $Q_{H}^{\star}(x,a)=F(x),$ and
\[
Q_{h}^{\star}(x,a):= R_{h}(x,a)+\mathbb{E}_{S_{h+1}\sim P_{h+1}%
(\cdot|x,a)}\left[  \sup_{a^{\prime}\in A}Q_{h+1}^{\boldsymbol{\star}}%
(S_{h+1},a^{\prime})\right], \quad h=H-1,\ldots,0.
\]
Then $Q_{h}^{\star}(x,a)$ is called the \textit{optimal state-action} function ($Q$-function) and one
thus has%
\[
V_{h}^{\boldsymbol{\star}}(x)=\sup_{a\in A}Q_{h}^{\star}(x,a),\text{ \ \ }%
\pi_{h}^{\boldsymbol{\star}}(x)\in\arg\max_{a\in\mathsf{A}}Q_{h}^{\star
}(x,a),\text{ \ \ }\quad h\in\lbrack H],
\]
provided the supremum is attainable in $\mathsf{A}$.
Finally, note that the optimal value function $V^{\star}$ satisfies due to Theorem~\ref{Bellman},
\[
V_{h}^{\star}(x)=T_{h}V_{h+1}^{\star}(x),\text{ \ \ }h\in\lbrack H[,
\]
where $T_{h}V(x):=\sup_{a\in A}\left(  R_{h}(x,a)+P_{h+1}^{a}V(x)\right)  $
with $P_{h+1}^{a}V(x):=\mathbb{E}_{S_{h+1}\sim P_{h+1}(\cdot
|x,a)}\left[  V(S_{h+1})\right].$

\begin{assumption}\label{defS}
For definiteness we henceforth assume that the state space $\mathsf{S}\subset\mathbb{R}^d$ for some
natural $d$,
and that the distribution \(P_{h+1}(dz
|x,a)\) possesses a Lebesgue density \(p_{h+1}^{a}(dz|x)\) for $x\in\mathsf{S}$ and  $a\in \mathsf{A}$. 
\end{assumption}
Let us
denote with $S_{h}\equiv(S_{h}(a_{<h}))_{h\in\{0,\ldots,H\}}$ the process
defined (in distribution) via
\begin{eqnarray}
\label{eq:chain}
S_{0}=x_0,\quad S_{h+1}\equiv S_{h+1}(a_{< h+1})\sim P_{h+1}(\cdot
|S_{h},a_{h}),\quad h=0,\ldots,H-1.
\end{eqnarray}
Then by Assumption~\ref{defS} the (unconditional) density of  $S_{h}$ denoted by $p_{h}^{a_{<h}}$  fulfills 
\begin{align}
p_{0}^{a_{<0}}(y)\,    =\delta_{x_{0}}(y), \quad
p_{h+1}^{a_{<h+1}}(y)\,    =\int_{\mathsf{S}}p_{h}^{a_{<h}}%
(z)p_{h+1}^{a_{h}}(y|z)\,dz,\quad  h\in\lbrack H[.\nonumber
\end{align}
\section{Algorithm}
\label{algB} 
Fix some ``representative''
 controls $b_{0},\ldots,b_{H-1}\in\mathsf{A}$ and simulate independently  for $n=1,\ldots,N,$ the
chains $\left(S_{h}^{(n)}= S_{h}^{(n)}(b_{<h})\right)  _{h\in\lbrack H]}$ according to
(\ref{eq:chain}) all starting from a fixed point \(x_0\in \mathsf{S}\). Fix some bounded function \(f\) on \(\mathsf{S}\) and consider the following approximation
\begin{align}
\mathbb{E}_{S_{h+1}\sim P_{h+1}(\cdot|x,a)}\left[  f
(S_{h+1})\right]    &  \approx \mathcal{E}_{h,N}(x,a;f):=
\sum_{n=1}^{N}f(S_{h+1}^{(n)})\,w_{h,n,N}%
(x,a) \text{ \ \ \ with}\nonumber\\
w_{h,n,N}(x,a)  &  :=\frac{p_{h+1}^{a}(S_{h+1}^{(n)}|x)}{\sum_{k=1,\,k\neq n}^{N}%
p_{h+1}^{b_{h}}(S_{h+1}^{(n)}|S_{h}^{(k)})}\Big/\sum_{n^{\prime}%
=1}^{N}\frac{p_{h+1}^{a}(S_{h+1}^{(n^{\prime})}|x)}{\sum_{k^{\prime
}=1,\,k^{\prime}\neq n^{\prime}}^{N}p_{h+1}^{b_{h}}(S_{h+1}^{(n^{\prime)
}}|S_{h}^{(k^{\prime})})}\label{appr}
\end{align}
for any \((x,a)\in (\mathsf{S}\times \mathsf{A})\) where by definition $0/0=0.$  Note that the weights $w_{l,n,N}$
satisfy  $w_{l,n,N}\geq0$ and $\sum_{n=1}^{N}w_{l,n,N}=1.$ 
The latter feature is crucial as it implies a contraction property of the corresponding approximated Bellman operator. 
A heuristic rationale 
behind (\ref{appr}) is given in Appendix~\ref{heurw}.
Let us observe that the weights in (\ref{appr}) are fundamentally different from the simpler weights used in \cite{rust1997using} and \cite{dufour2015approximation}. A similar choice cannot be applied here because of the finite horizon setting and the absence  of a control-independent reference measure in our  context. Therefore, we propose the following (pseudo) weighted stochastic mesh algorithm.

\begin{algorithm}
\begin{itemize}
\item Initialization: $\overline{V}_{H}(S_{H}^{(n)})=F(S_{H}^{(n)}),$
$n=1,\ldots,N.$

\item Backward step: Suppose that for $h+1\leq H,$ $\overline{V}_{h+1}%
(S_{h+1}^{(n)})$ is constructed for $n=1,\ldots,N.$ Then we define
\begin{gather}
\label{eq:approx-value}
\overline{V}_{h}(S_{h}^{(r)})=
\sup_{a\in\mathsf{A}}\left(  R_{h}(  S_{h}^{(r)},a)
+\mathcal{E}_{h,N}(S_{h}^{(r)},a;\overline{V}_{h+1})\right) 
\end{gather}
for $r=1,\ldots,N.$
\item Output: \(\overline{V}_{0}(x_0).\)
\end{itemize}
\end{algorithm}
Note that the above algorithm depends on the choice of controls $b_{0},\ldots,b_{H-1}\in\mathsf{A}$. However, as we show in the next section, this choice of controls doesn't influence the convergence rates of the algorithm under proper assumptions.
\section{Convergence analysis}
In this section we study the convergence of the stochastic mesh algorithm presented in Section~\ref{algB}.  First we assume that the state/action space is compact and then extend our results to a noncompact case. Throughout this section we make the following assumption. 
\begin{assumption}
\label{ass:sets} Assume that  \(\mathsf{S}\subset \mathbb{R}^{d_\mathsf{S}}\) and \(\mathsf{A}\subset \mathbb{R}^{d_\mathsf{A}}\) for some natural numbers \(d_\mathsf{S}\) and \(d_\mathsf{A}.\)  Moreover, \(\mathsf{S}\) and \(\mathsf{A}\) are compact with
(finite)
diameters \(\mathrm{diam}(\mathsf{S})\) and \(\mathrm{diam}(\mathsf{A}),\) respectively.
\end{assumption}
\begin{assumption}
\label{ass: reg} There exist  constants $\delta>0,$ \(\Lambda>0\) and a function \(\mathcal{L}:\) \(\mathsf{S}\to \mathbb{R}_+\) such that the one-step
transition densities \((p_{h}^{a}(y|x),\,h\in [H])\) satisfy
\[
0<\delta\leq p_{h}^{a}(y|x)\leq\Lambda,\quad |p_{h}^{a_1}(y|x_1)-p_{h}^{a_2}(y|x_2)|\leq \mathcal{L}(y) (|x_1-x_2|+\rho_{\mathsf{A}}(a_1,a_2) )
\]
for all \(x,x_1,x_2,y\in \mathsf{S},\) \(a,a_1,a_2\in \mathsf{A}\) and \(h=1,\ldots,H,\) where  \(\max\{\|\mathcal{L}\|_{L^1(\mathsf{S})},\|\mathcal{L}\|_{L^\infty(\mathsf{S})}\}\leq L.\) Moreover
\[
\max\{|R_{h}(s,a)|,|F(s)|\}\leq G,\quad (s,a)\in \mathsf{S}\times \mathsf{A},\quad h\in [H[.
\]
\end{assumption}
Under these assumptions, we can prove the following bound.
\begin{theorem}\label{thmcomp}
With respect to the probability space supporting 
the simulations in algorithm (\ref{eq:approx-value}), it  holds that%
\begin{align*}
\mathbb{E}\left[  \left\vert \overline{V}_{0}(x_0)-V_{0}^{\star}(x_0)\right\vert
\right]   &  \lesssim\frac{H^2
G}{\sqrt{N}}\left(  \frac{L\, \mathrm{DI}(\mathsf{S}%
\times\mathsf{A}) +   L\,\mathrm{diam}(\mathsf{S})\mathrm{diam}(\mathsf{A})   +  \Lambda }{\delta}+\frac{\Lambda^2}{\delta^{2}}\right)  
\end{align*}
for all $N>N_{0}$ with $N_{0}$ large enough  and $\lesssim$ denoting $\leq$ up to
some (absolute)  proportionality constant. 
\end{theorem}
\begin{proof}
For $r=1,\ldots,N,$ one has
\begin{align*}
&  \left\vert \overline{V}_{h}(S_{h}^{(r)})-V_{h}^{\star}(S_{h}^{(r)})\right\vert \\
&  \leq\sup_{a\in\mathsf{A}}\left\vert \sum_{n=1}^{N}\overline{V}%
_{h+1}(S_{h+1}^{(n)})w_{h,n,N}(S_{h}^{(r)},a)-\mathbb{E}_{S_{h+1}\sim P_{h+1}(\cdot|S_{h}^{(r)},a)}\left[  V_{h+1}^{\star}(S_{h+1})\right]  \right\vert \\
&  \leq\sup_{a\in\mathsf{A}}\sum_{n=1}^{N}\left\vert \overline{V}%
_{h+1}(S_{h+1}^{(n)})-V_{h+1}^{\star}(S_{h+1}^{(n)})\right\vert
w_{h,n,N}(S_{h}^{(r)},a)\\
&  +\sup_{a\in\mathsf{A}}\left\vert \sum_{n=1}^{N}V_{h+1}^{\star}(S_{h+1}^{(n)%
})w_{h,n,N}(S_{h}^{(r)},a)-\mathbb{E}_{S_{h+1}\sim P_{h+1}(\cdot|S_{h}^{(r)},a)}\left[  V_{h+1}^{\star}(S_{h+1})\right]
\right\vert \\
&  \leq \|\overline{V}_{l+1}-V_{l+1}^{\star}\|_{N}+\mathcal{R}%
_{l+1},
\end{align*}
where%
\[
\|\overline{V}_{h}-V_{h}^{\star}\| _{N}:=\max_{1\leq r\leq
N}\left\vert \overline{V}_{h}(S_{h}^{(r)})-V_{h}^{\star}(S_{h}^{(r)%
})\right\vert 
\]
and
\[
\mathcal{R}_{h+1}:=
\sup_{\substack{a\in\mathsf{A}\\1\leq r\leq N}}\left\vert \sum_{n=1}%
^{N}V_{h+1}^{\star}(S_{h+1}^{(n)})w_{h,n,N}(S_{h}^{(r)},a)-\mathbb{E}_{S_{h+1}\sim P_{h+1}(\cdot|S_{h}^{(r)},a)}\left[  V_{h+1}^{\star}(S_{h+1})\right] \right\vert .
\]
Hence, since $\overline{V}_{H}-V_{H}^{\star}=0,$
\begin{equation}
\| \overline{V}_{h}-V_{h}^{\star}\|_{N}\leq\| \overline
{V}_{h+1}-V_{h+1}^{\star}\| _{N}+\mathcal{R}_{h+1}\leq\sum_{k=h}%
^{H-1}\mathcal{R}_{k+1}.\label{sumR}%
\end{equation}
We now proceed with the estimation of $\mathbb{E}\left[  \mathcal{R}%
_{h+1}\right]  ,$ $h=0,\ldots,H-1,$ and write%
\[
\sum_{n=1}^{N}V_{h+1}^{\star}(S_{h+1}^{(n)})w_{h,n,N}(S_{h}^{(r)}%
,a)-\mathbb{E}_{S_{h+1}\sim P_{h+1}(\cdot|S_{h}^{(r)},a)}\left[
V_{h+1}^{\star}(S_{h+1})\right]  =T_{1,r}(a)+T_{2,r}(a)+T_{3,r}(a)
\]
where
\begin{align*}
T_{1,r}(a) &  :=\sum_{n=1}^{N}{V}_{h+1}^{\star}(S_{h+1}^{(n)})\frac
{p_{h+1}^{a}(S_{h+1}^{(n)}|S_{h}^{(r)})}{\sum_{k=1,\,k\neq n}^{N}%
p_{h+1}^{b_{h}}(S_{h+1}^{(n)}|S_{h}^{(k)})}\left(  \left[  \sum_{n^{\prime}%
=1}^{N}\frac{p_{h+1}^{a}(S_{h+1}^{(n^{\prime})}|S_{h}^{(r)})}{\sum_{k^{\prime
}=1,\,k^{\prime}\neq n^{\prime}}^{N}p_{h+1}^{b_{h}}(S_{h+1}^{(n^{\prime}%
)}|S_{h}^{(k^{\prime})})}\right]  ^{-1}-1\right)  ,\\
T_{2,r}(a) &  :=\sum_{n=1}^{N}{V}_{h+1}^{\star}(S_{h+1}^{(n)})\left(
\frac{p_{h+1}^{a}(S_{h+1}^{(n)}|S_{h}^{(r)})}{\sum_{k=1,\,k\neq n}^{N}%
p_{h+1}^{b_{h}}(S_{h+1}^{(n)}|S_{h}^{(k)})}-\frac{p_{h+1}^{a}(S_{h+1}%
^{(n)}|S_{h}^{(r)})}{Np_{h+1}^{_{b_{<h+1}}}(S_{h+1}^{(n)})}\right)  ,\\
T_{3,r}(a) &  :=\sum_{n=1}^{N}{V}_{h+1}^{\star}(S_{h+1}^{(n)})\frac
{p_{h+1}^{a}(S_{h+1}^{(n)}|S_{h}^{(r)})}{Np_{h+1}^{_{b_{<h+1}}}(S_{h+1}%
^{(n)})}-\mathbb{E}_{S_{h+1}\sim P_{h+1}(\cdot|S_{h}^{(r)},a)}\left[
V^\star_{h+1}(S_{h+1})\right]  .
\end{align*}
We have
\begin{align*}
|T_{1,r}(a)| &  \leq HG\sup_{x\in\mathsf{S},\,a\in\mathsf{A}}\left\vert
\sum_{n^{\prime}=1}^{N}\frac{p_{h+1}^{a}(S_{h+1}^{(n^{\prime})}|x)}%
{\sum_{k^{\prime}=1,\,k^{\prime}\neq n^{\prime}}^{N}p_{h+1}^{b_{h}}%
(S_{h+1}^{(n^{\prime})}|S_{h}^{(k^{\prime})})}-1\right\vert \\
&  \leq HG\sup_{x\in\mathsf{S},\,a\in\mathsf{A}}\left\vert 1-\frac{1}{N}%
\sum_{n=1}^{N}\frac{p_{h+1}^{a}(S_{h+1}^{(n)}|x)}{p_{h+1}^{_{b_{<h+1}}%
}(S_{h+1}^{(n)})}\right\vert \\
&  +\frac{HG\Lambda}{N}\sum_{n=1}^{N}\left\vert \frac{1}{p_{h+1}^{_{b_{<h+1}}%
}(S_{h+1}^{(n)})}-\frac{1}{\frac{1}{N}\sum_{k=1,\,k\neq n}^{N}p_{h+1}^{b_{h}%
}(S_{h+1}^{(n)}|S_{l}^{(k)})}\right\vert .
\end{align*}
Note that
\[
\mathbb{E}\left[  \frac{p_{h+1}^{a}(S_{h+1}^{(n)}|x)}{p_{h+1}^{_{b_{<h+1}}%
}(S_{h+1}^{(n)})}\right]  =\int_{\mathsf{S}}p_{h+1}^{a}(z|x)\,dz=1,
\]
and
\[
p_{h+1}^{a_{<h+1}}(y)\,=\int_{\mathsf{S}}p_{h}^{a_{<h}}(z)p_{h+1}^{a_{h}%
}(y|z)\,dz\geq\delta\int_{\mathsf{S}}p_{h}^{a_{<h}}(z)\,dz=\delta
\]
for all $y\in\mathsf{S},$ $h\in\lbrack H[,$ $a_{<h+1}\in\mathsf{A}^{l+1}.$ It
then follows by Proposition~\ref{prop:unif-exp}, that%

\[
\mathbb{E}\sup_{x\in\mathsf{S},a\in\mathsf{A}}\left\vert 1-\frac{1}{N}%
\sum_{n=1}^{N}\frac{p_{h+1}^{a}(S_{h+1}^{(n)}|x)}{p_{h+1}^{_{b_{<h+1}}%
}(S_{h+1}^{(n)})}\right\vert \lesssim\frac{L\,\mathrm{DI}(\mathsf{S}%
\times\mathsf{A})+L\,\mathrm{diam}(\mathsf{S}\times\mathsf{A})+\Lambda
}{\delta\sqrt{N}}.
\]
Furthermore,
\begin{align*}
&  \mathbb{E}_{S_{h+1}^{(1)}}\left[  \left\vert \frac{1}{p_{h+1}^{_{b_{<h+1}}%
}(S_{h+1}^{(1)})}-\frac{1}{\frac{1}{N}\sum_{k=1,\,k\neq1}^{N}p_{h+1}^{b_{h}%
}(S_{h+1}^{(1)}|S_{h}^{(k)})}\right\vert \right]  \\
&  \leq\frac{N}{N-1}\delta^{-2}\mathbb{E}_{S_{h+1}^{(1)}}\left[  \left\vert
\frac{1}{N}\sum_{k=1,\,k\neq1}^{N}p_{h+1}^{b_{h}}(S_{h+1}^{(1)}|S_{h}%
^{(k)})-p_{h+1}^{_{b_{<h+1}}}(S_{h+1}^{(1)})\right\vert \right]  \\
&  \leq\delta^{-2}\mathbb{E}_{S_{h+1}^{(1)}}\left[  \left\vert \frac{1}%
{N-1}\sum_{k=1,\,k\neq1}^{N}\left(  p_{h+1}^{b_{h}}(S_{h+1}^{(1)}|S_{h}%
^{(k)})-p_{h+1}^{_{b_{<h+1}}}(S_{h+1}^{(1)})\right)  \right\vert \right]  \\
&  +\frac{\delta^{-2}}{N-1}p_{h+1}^{_{b_{<h+1}}}(S_{h+1}^{(1)})\\
&  \leq\frac{\delta^{-2}}{\sqrt{N-1}}\sqrt{\int_{\mathsf{S}}p_{h+1}^{b_{h}%
}(S_{h+1}^{(1)}|z)^{2}p_{h}^{_{b_{<h}}}(z)\,dz}+\frac{\delta^{-2}}{N-1}%
p_{h+1}^{_{b_{<h+1}}}(S_{h+1}^{(1)}),
\end{align*}
and so for each $n=1,\ldots,N$ we have by symmetry and Jensen's inequality,
\begin{multline*}
\mathbb{E}\left[  \left\vert \frac{1}{p_{h+1}^{_{b_{<h+1}}}(S_{h+1}^{(n)}%
)}-\frac{1}{\frac{1}{N}\sum_{k=1,\,k\neq n}^{N}p_{h+1}^{b_{h}}(S_{h+1}%
^{(n)}|S_{h}^{(k)})}\right\vert \right]  \\
=\mathbb{E}\left[  \left\vert \frac{1}{p_{h+1}^{_{b_{<h+1}}}(S_{h+1}^{(1)}%
)}-\frac{1}{\frac{1}{N}\sum_{k=1,\,k\neq1}^{N}p_{h+1}^{b_{h}}(S_{h+1}%
^{(1)}|S_{h}^{(k)})}\right\vert \right]  \\
\lesssim\frac{\delta^{-2}}{\sqrt{N}}\sqrt{\int_{\mathsf{S}\times\mathsf{S}%
}p_{h+1}^{b_{h}}(z^{\prime}|z)^{2}p_{h}^{_{b_{<h}}}(z)p_{h+1}^{_{b_{<h+1}}%
}(z^{\prime})\,dz\,dz^{\prime}}
\lesssim\frac{\Lambda}{\delta^{2}\sqrt{N}}%
\end{multline*}
for $N>N_{0}.$ Analogously, we have
\begin{align*}
\mathbb{E}\left[  \max_{r\in\lbrack N],\,a\in\mathsf{A}}|T_{2,r}(a)|\right]
&  \leq HG\Lambda\mathbb{E}\left[  \left\vert \frac{1}{\frac{1}{N}%
\sum_{k=1,\,k\neq1}^{N}p_{h+1}^{b_{h}}(S_{h+1}^{(1)}|S_{h}^{(k)})}-\frac
{1}{p_{h+1}^{_{b_{<h+1}}}(S_{h+1}^{(1)})}\right\vert \right]  \\
&  \leq\frac{HG\Lambda^{2}}{\delta^{2}\sqrt{N}}%
\end{align*}
for $N>N_{0}.$ We next consider $T_{3,r}.$ For each
fixed $x\in \mathsf{S}$ and $a\in\mathsf{A},$ we have
\[
\mathbb{E}\left[  {V}_{h+1}^{\star}(S_{h+1}^{(n)})\frac{p_{h+1}^{a}%
(S_{h+1}^{(n)}|x)}{p_{h+1}^{_{b_{<h+1}}}(S_{h+1}^{(n)})}\right]
=\int_{\mathsf{S}}{V}_{h+1}^{\star}(z)p_{h+1}^{a}(z|x)\,dz=\mathbb{E}%
_{S_{h+1}\sim P_{h+1}(\cdot|x,a)}\left[  V_{h+1}^{\star}(S_{h+1})\right]  .
\]
Then, by  Proposition~\ref{prop:unif-exp} again, it follows that
\begin{multline*}
\mathbb{E}\left[  \max_{r\in\lbrack N],\,a\in\mathsf{A}}|T_{3,r}(a)|\right]
\leq\\
\mathbb{E}\left[  \sup_{x\in\mathsf{S},\,a\in\mathsf{A}}\left\vert \frac{1}%
{N}\sum_{n=1}^{N}{V}_{h+1}^{\star}(S_{h+1}^{(n)})\frac{p_{h+1}^{a}%
(S_{h+1}^{(n)}|x)}{p_{h+1}^{_{b_{<h+1}}}(S_{h+1}^{(n)})}-\mathbb{E}%
_{S_{h+1}\sim P_{h+1}(\cdot|x,a)}\left[  V_{h+1}^{\star}(S_{h+1})\right]
\right\vert \right]  \\
\lesssim\frac{HG}{\delta}\frac{L\,\mathrm{DI}(\mathsf{S}\times\mathsf{A}%
)+L\,\mathrm{diam}(\mathsf{S}\times\mathsf{A})+\Lambda}{\sqrt{N}}.
\end{multline*}
Finally,  we apply (\ref{sumR}) for $h=0.$
\end{proof}
\subsection{Non-compact case}
If \(\mathsf{S}\) is not compact subset of \(\mathbb{R}^d\)  we consider its approximation by compact susbsets.
Let \(\mathcal{D}\) be compact subset of \(\mathsf{S}\) and let $(S_{k}^{(h,x),\mathcal{D}}(\boldsymbol{\pi}),\,k=h,\ldots,H)$ be a
process obtained by reflection of the chain $(S_{k}^{(h,x)}(\boldsymbol{\pi
}),\,k=h,\ldots,H)$ reflected in  $\mathcal{D}$ as described
in Appendix~\ref{sec:refl}. For a fixed policy $\boldsymbol{\pi}\in\Pi$ and
$x\in\mathcal{D},$ consider the exit (stopping) times,%
\[
\tau_{h}^{x,\mathcal{D}}:=\min\left\{  k\geq h:S_{k}^{(h,x)}\notin
\mathcal{D}\right\}  .
\]
where $(S_{k}^{(h,x)}=S_{k}^{(h,x)}(\boldsymbol{\pi}),\,k=h,\ldots,H)$ stands
for the chain \eqref{eq:chain} following the policy $\boldsymbol{\pi}$ and
starting in $x$ at time $h.$ Hence%
\[
S_{k}^{(h,x),\mathcal{D}}1_{\left\{  \tau_{h}^{x,\mathcal{D}}>k\right\}
}\overset{\mathrm{Law}}{=}S_{k}^{(h,x)}1_{\left\{  \tau_{h}^{x,\mathcal{D}%
}>k\right\}  }.
\]
We now consider the MDP in the compact domain $\mathcal{D},$
\begin{equation}
{V}_{h}^{\mathcal{D}}(x):=\sup_{\boldsymbol{\pi}\in\Pi}\mathbb{E}%
_{\boldsymbol{\pi}}\left[  \sum_{k=h}^{H-1}R_{k}(S_{k}^{(h,x),\mathcal{D}%
},A_{k})+F_{H}(S_{H}^{(h,x),\mathcal{D}})\right]  \label{MDPt}%
\end{equation}
as an approximation to $V_{h}(x),$ $h\in\lbrack H-1].$ It is not difficult to
see that
\begin{align*}
\left\vert {V}_{h}^{\mathcal{D}}(x)-V_{h}^{\star}(x)\right\vert  & \lesssim
HG\sup_{\boldsymbol{\pi}\in\Pi}\mathbb{P}_{\boldsymbol{\pi}}(\tau
_{h}^{x,\mathcal{D}}\leq H)\\
& \lesssim HG\sum_{l=h}^{H}\sup_{\boldsymbol{\pi}\in\Pi}\mathbb{P}%
_{\boldsymbol{\pi}}\bigl(S_{l}^{(x,h)}\notin\mathcal{D}\bigr).
\end{align*}
Furthermore, the one-step transition density $p_{h}^{a,\mathcal{D}}$ of the
process $(S_{h}^{\mathcal{D}})$ is given by (see Appendix~\ref{sec:refl})%
\[
p_{h}^{a,\mathcal{D}}(y|x)=p_{h}^{a}(y|x)+\frac{1}{\lambda(\mathcal{D})}%
\int_{\mathbb{R}^{d}\setminus\mathcal{D}}p_{h}^{a}(z|x)dz,\text{ \ \ }%
x,y\in\mathcal{D}.
\]
Instead of  Assumption~\ref{ass: reg} we now consider the following weaker assumption on \(\mathcal{D}\).
\begin{assumption}
\label{ass: reg_unbound} For any compact subset $\mathcal{D}$ of $\mathsf{S}$
there exist some constants $\delta_{\mathcal{D}}>0,$ $\Lambda>0,$
$L_{\mathcal{D}}>0,$ and a function $\mathcal{L}_{\mathcal{D}}:$
$\mathsf{S}\rightarrow\mathbb{R}_{+}$ such that the 
one-step transition
density $p_{h}^{a}$ satisfies
\[
0<\delta_{\mathcal{D}}\leq p_{h}^{a}(z|x)\leq\Lambda,\quad\left\vert
p_{h}^{a_{1}}(y|x_{1})-p_{h}^{a_{2}}(y|x_{2})\right\vert \leq\mathcal{L}%
_{\mathcal{D}}(y)(|x_{1}-x_{2}|+\rho_{\mathsf{A}}(a_{1},a_{2}))
\]
for all $x,z,x_{1},x_{2}\in\mathcal{D},$ 
$y\in\mathbb{R}^{d},$ $a,a_{1},a_{2}%
\in\mathsf{A}$ and $h=1,\ldots,H,$ where
\[
\Vert\mathcal{L}_{\mathcal{D}}\Vert_{L^{\infty}(\mathcal{D})}+\frac{1}%
{\lambda\left(  \mathcal{D}\right)  }\Vert\mathcal{L}_{\mathcal{D}}%
\Vert_{L^{1}(\mathbb{R}^{d}\setminus\mathcal{D})}\leq L_{\mathcal{D}}.
\]
Moreover
\[
\max\{|R_{h}(s,a)|,|F(s)|\}\leq G,\quad(s,a)\in\mathsf{S}\times\mathsf{A}%
,\quad h\in\lbrack H[.
\]

\end{assumption}
Hence  \(p_{h}^{a,\mathcal{D}}(y|x)\geq\delta_{\mathcal{D}}\)
 for $x,y\in
\mathcal{D}$ and \(a\in  \mathsf{A}\), and furthermore,
\begin{eqnarray*}
|p_{h}^{a_{1},\mathcal{D}}(y|x_{1})-p_{h}^{a_{2},\mathcal{D}}(y|x_{2}%
)|&\leq &|p_{h}^{a_{1}}(y|x_{1})-p_{h}^{a_{2}}(y|x_{2})|+\frac{1}{\lambda
(\mathcal{D})}\int_{\mathbb{R}^{d}\setminus\mathcal{D}}\left\vert p_{h}%
^{a_{1}}(z|x_{1})-p_{h}^{a_{2}}(z|x_{2})\right\vert dz
\\
&\leq & L_{\mathcal{D}}(|x_1-x_2|+\rho_{\mathsf{A}}(a_1,a_2) ).
\end{eqnarray*}
\begin{theorem}
\label{thm:bound-ineq}
Fix some \(x_0\) then under Assumption~\ref{ass: reg_unbound}, it  holds that%
\begin{align*}
\mathbb{E}\left[  \bigl|\overline{V}_{0}^{\mathcal{D}}(x_0)-V_{0}^{\star}(x_0)\bigr |
\right]   &  \lesssim\frac{H^2
G}{\sqrt{N}}\left(  \frac{L_{\mathcal{D}} \, \mathrm{DI}(\mathcal{D}\times\mathsf{A}%
) +   L_{\mathcal{D}}\,\mathrm{diam}(\mathcal{D})\mathrm{diam}(\mathsf{A})   +  \Lambda_{\mathcal{D}} }{\delta_{\mathcal D}}+\frac{\Lambda^2
}{\delta_{\mathcal D}^{2}}\right)
\\
&+HG\sum_{l=0}^{H}\sup_{\boldsymbol{\pi}\in\Pi}\mathbb{P}%
_{\boldsymbol{\pi}}\bigl(S_{l}^{(x_0,0)}\notin\mathcal{D}\bigr)  
\end{align*}
for all $N>N_{0}$ with $N_{0}$ large enough  and $\lesssim$ denoting $\leq$ up to
some absolute proportionality constant. 
\end{theorem}

\section{Complexity}
In this section, we estimate the computational budget, that is, the complexity, needed for computing $V^\star(x_0)$ (in $L^1$) with a given accuracy $\epsilon>0$, by the algorithm presented in Section~\ref{algB}. For simplicity we disregard the cost of the optimization step and identify the overall
cost with $HN^2$, that is, the costs of computing all weights \eqref{appr} for $N$ trajectories.
\subsection{Complexity for compact $\mathsf{S}$ and $\mathsf{A}$}\label{compcomp}
In order to explicitly incorporate the dimension of the state and
action space in the complexity estimation, we consider a sequence of MDPs for
running $d=1,2,\ldots$ Without much loss of generality we assume that
$\mathsf{S}_{d}=B_{R_{d}}^{d}\subset\mathbb{R}^{d},$ $\mathsf{A}_{d}=B_{A_{d}%
}^{d}\subset\mathbb{R}^{d}$ for some $R_{d}>0$ and $A_{d}>0$ with $B_{R}^{d}$
being the Euclidean ball in $\mathbb{R}^{d}$ of radius $R.$ We further assume
that in dimension $d,$ the transition probabilities are given by $p_{d,h}%
^{a}(y|x).$ Furthermore it is assumed that the bound $G$ in Assumption~\ref{ass: reg}
holds uniformly in $d.$ Obviously, if $R_{d},A_{d},$ and $p_{d,h}^{a}(y|x)$
are such that $L_{d},$ $\Lambda_{d},$ and $\delta_{d}^{-1}$ due to
Assumption~\ref{ass: reg} can taken to be polynomially bounded in $d,$ then Theorem~\ref{thmcomp} implies that%
\[
\mathcal{C}\left(  \epsilon,d\right)  \lesssim\frac{H^{9}G^{4}}{\epsilon^{4}%
}\text{polynomial}\left(  d\right)
\]
In this case, the mesh algorithm is tractable in the sense of \cite{NovakWozniakowski2008}, that is,%
\[
\lim_{d+\epsilon^{-1}\rightarrow\infty}\frac{\log\mathcal{C}\left(
\epsilon,d\right)  }{d+\epsilon^{-1}}=0.
\]
This result can be seen as an extension of \cite{rust1997using} to the case of finite horizon MDPs with more general state and action spaces. 
\subsection{Complexity  for noncompact $\mathsf{S}$ and compact $\mathsf{A}$. In particular, $\mathsf{A}$ can be infinite.}

Let us now consider the noncompact case with $\mathsf{S}=\mathbb{R}^{d},$
$A\subset\mathbb{R}^{d}$ in the setup of Section~\ref{compcomp}. We then have
the following result.

\begin{proposition}
\label{prop:noncomp-compl} Suppose that, for a
{generic} $d,$ there is a sequence of compact sets $\mathcal{D}_{d,N},$
$N\in\mathbb{N}$, such that
\[
\frac{L_{\mathcal{D}_{d,N}}\,\mathrm{DI}(\mathcal{D}_{d,N}\times\mathsf{A}%
_{d})+L_{\mathcal{D}_{d,N}}\,\mathrm{diam}(\mathcal{D}_{d,N})\mathrm{diam}%
(\mathsf{A}_{d})}{\delta_{\mathcal{D}_{d,N}}}+\frac{\Lambda_{d}^{2}}%
{\delta_{\mathcal{D}_{d,N}}^{2}}\leq C_{1}(H,\log N,d)N^{\alpha}%
\]
and
\[
\sum_{h=1}^{H}\sup_{\boldsymbol{\pi}\in\Pi}\mathbb{P}_{\boldsymbol{\pi}%
}\bigl(S_{h}^{(x_{0},0)}\notin\mathcal{D}_{d,N}\bigr)\leq C_{2}(H,\log
N,d)N^{-\beta}%
\]
for $N>N_{0}$, where $C_{1}$ and $C_{2}$ are functions on $\mathbb{N}%
\times\mathbb{R}\times\mathbb{N}$ such that
\[
0\leq C_{1,2}(x,y,d)\leq c_{d}\left\vert xy\right\vert ^{q_{d}}\text{ \ \ for
all }x,y\geq1,
\]
and the parameters $\alpha\in\lbrack0,1/2),$ $\beta>0$ do not depend on $N$
and $d$. Here both $c_{d}>0$ and $q_{d}\in\mathbb{R}_{+}$ are independent of
$H$ and $\epsilon.$ Then the complexity $\mathcal{C}(\epsilon,d)$ of our
algorithm can be bounded as
\begin{multline}
\mathcal{C}(\epsilon,d)\lesssim H\max\left(  2GcH^{2}\left(  \frac
{2H}{1-2\alpha}\right)  ^{q_{d}},2GcH\left(  \frac{H}{\beta}\right)  ^{q_{d}%
},1\right)  ^{2\max\left(  1/\beta,2/(1-2\alpha)\right)  }\label{noncompcomp0}%
\\
\times\frac{\log^{2q_{d}\max\left(  1/\beta,2/(1-2\alpha)\right)  }%
(1/\epsilon)}{\epsilon^{2\max\left(  1/\beta,2/(1-2\alpha)\right)  }}.
\end{multline}

\end{proposition}

\begin{corollary}
\label{addcor} If one has in addition that $q_{d}\leq\eta d$ and $c_{d}\leq
c_{0}\exp\left(  \lambda d\right)  $ for some universal constants
$c_{0},\eta,\lambda>0,$ one obtains%
\begin{multline*}
\mathcal{C}\left(  \epsilon,d\right)  \lesssim H\max\left(  2Gc_{0}e^{\lambda
d}H^{2+\eta d}\frac{2^{\eta d}}{\left(  1-2\alpha\right)  ^{\eta d}}%
,2Gc_{0}e^{\lambda d}H^{1+\eta d}\frac{1}{\left(  \beta\wedge1\right)  ^{\eta
d}},1\right)  ^{2\max\left(  1/\beta,2/(1-2\alpha)\right)  }\\
\times\frac{\log^{2\eta d\max\left(  1/\beta,2/(1-2\alpha)\right)
}(1/\epsilon)}{\epsilon^{2\max\left(  1/\beta,2/(1-2\alpha)\right)  }},
\end{multline*}
which implies%
\[
\log\mathcal{C}\left(  \epsilon,d\right)  =r_{1}\log H+\left(  r_{2}+r_{3}\log
H+r_{4}\log\log\frac{1}{\epsilon}\right)  d+r_{5}\log\frac{1}{\epsilon}%
\]
for certain constants $r_{1},\ldots,r_{5}>0.$ From this it is easy to see that
the problem is not tractable in the sense of \cite{NovakWozniakowski2008},
but, since%
\[
\lim_{d\rightarrow\infty}\lim_{\epsilon\searrow0}\frac{\log\mathcal{C}\left(
\epsilon,d\right)  }{f(d)\log(1/\epsilon)}=0\text{ \ \ for any }f\text{ with
}f(d)\rightarrow\infty\text{ as }d\rightarrow\infty,
\]
the problem is semi-tractable in the sense of \cite{belkalsch} and we have a
kind of \textquotedblleft weak curse of dimensionality\textquotedblright.
\end{corollary}

The next section provides an example where Corollary~\ref{addcor} applies.

\subsection{Example: Gaussian transition densities}\label{GaussEx}
Let us consider the case of Gaussian transition probabilities of the form
\begin{equation}
p_{h}^{a}(y|x)\equiv
p_{d,h}^{a}(y|x)=\frac{1}{(2\pi\sigma_{h}^{2})^{d/2}}\exp(-|x-y-a|^{2}%
/(2\sigma_{h}^{2})),\quad x,y\in\mathbb{R}^{d}%
\label{eq:gauss-trans-dens}
\end{equation}
where the (scalar) variances $\sigma_{h},$ $h\in ]H]$ are all bounded
from above and below, that is, $0<\sigma_{\min}\leq\sigma_{h}\leq\sigma_{\max
}<\infty.$ Let $\mathsf{S}=\mathbb{R}^{d},$ $\mathsf{A}_d=B^d_{A}\subset
\mathbb{R}^{d}$ for some $A>0$ and $\mathcal{D}_{d,N}=B^d_{R_{N}}$ with $B^d_{R}$
being the Euclidean ball in $\mathbb{R}^{d}$ of radius $R.$ 
Such densities naturally appear as transition densities of discretized (e.g. via Euler scheme) diffusion processes, see \ref{sec:lqg} for numerical illustrations. Let us check now the assumptions of Proposition~\ref{prop:noncomp-compl} and Corollary~\ref{addcor}.
In what follows, we do not always denote  dependence on $d$ explicitly, for notational convenience. 
Choosing
\begin{equation}
R_{N}=\sqrt{\gamma\sigma_{\min}^{2}\log(N)/4}\text{ \ \ for some \ }\gamma
\in(0,1/4),\label{LNa}%
\end{equation}
we see that
\begin{equation}
p_{h}^{a}(y|x)\geq\frac{1}{(2\pi\sigma_{\max}^{2})^{d/2}}\exp(-A^{2}%
/\sigma_{\min}^{2})N^{-\gamma}=:\delta_{\mathcal{D}_{N}}\label{LN0}%
\end{equation}
for all $a\in\mathsf{A},$ $x,y\in\mathcal{D}_{N}.$ Furthermore, we have
\begin{equation}
p_{h}^{a}(y|x)\leq(2\pi\sigma_{\min}^{2})^{-d/2}=:\Lambda\text{ \ \ for all
\ \ }h\in\lbrack H].\label{LN00}%
\end{equation}
Note that for all $x_{1},x_{2}\in\mathcal{D}_N,$ $x,y\in\mathbb{R}^{d},$
$a,a_{1},a_{2}\in\mathsf{A},$ $h\in]H],$
\begin{align*}
\left\vert \nabla_{x}p_{h}^{a}(y|x)\right\vert  & =\left\vert \nabla_{a}%
p_{h}^{a}(y|x)\right\vert \\
& =\frac{1}{\sigma_{h}^{2}(2\pi\sigma_{h}^{2})^{d/2}}\left\vert
x-a-y\right\vert \exp(-|x-y-a|^{2}/(2\sigma_{h}^{2})).
\end{align*}
Hence
\begin{align*}
& \left\vert p_{h}^{a_{1}}(y|x_{1})-p_{h}^{a_{2}}(y|x_{2})\right\vert \\
& \leq\frac{\sqrt{2}}{\sigma_{\min}^{d+2}(2\pi)^{d/2}}\sup_{x\in
\mathcal{D}_{N},a\in\mathsf{A}\,}\left\{  \left\vert x-a-y\right\vert
\exp(-|x-y-a|^{2}/(2\sigma_{\max}^{2}))\right\}  \\
& \times\left(  \left\vert x_{1}-x_{2}\right\vert +\left\vert a_{1}%
-a_{2}\right\vert \right)  \\
& =:\mathcal{L}_{\mathcal{D}_{N}}\left(  y\right)  \left(  \left\vert
x_{1}-x_{2}\right\vert +\left\vert a_{1}-a_{2}\right\vert \right)  .
\end{align*}
So on the one hand we have%
\begin{align*}
\left\Vert \mathcal{L}_{\mathcal{D}_{N}}\right\Vert _{L^{\infty}\left(
\mathcal{D}_{N}\right)  }  & \leq\frac{\sqrt{2}}{\sigma_{\min}^{d+2}%
(2\pi)^{d/2}}\left(  2R_{N}+A\right)  \\
& \simeq\frac{R_{N}}{\sigma_{\min}^{d+2}2^{(d-3)/2}\pi^{d/2}},\text{
\ \ }N\rightarrow\infty.
\end{align*}
On the other hand, for $R_{N}>A$ and $\left\vert y\right\vert \geq2R_{N}$ it
holds that%
\[
0\leq\mathcal{L}_{\mathcal{D}_{N}}\left(  y\right)  \leq\frac{\sqrt{2}}%
{\sigma_{\min}^{d+2}(2\pi)^{d/2}}3\left\vert y\right\vert \exp(-\left(
|y|-2R_{N}\right)  ^{2}/(2\sigma_{\max}^{2})),
\]
from which we see that $\left\Vert \mathcal{L}_{\mathcal{D}_{N}}\right\Vert
_{L^{1}\left(  \mathbb{R}^{d}\right)  }<\infty$ for any $N,$ and moreover%
\begin{align*}
\left\Vert \mathcal{L}_{\mathcal{D}_{N}}\right\Vert _{L^{1}\left(
\mathbb{R}^{d}\right)  }  & \leq\left\Vert \mathcal{L}_{\mathcal{D}_{N}%
}\right\Vert _{L^{\infty}\left(  B_{R_{2N}}\right)  }\mathsf{Vol}\left(
B_{R_{2N}}\right)  \\
& +\frac{\sqrt{2}}{\sigma_{\min}^{d+2}(2\pi)^{d/2}}\int_{\left\vert
y\right\vert \geq2R_{N}}3\left\vert y\right\vert \exp(-\left(  |y|-2R_{N}%
\right)  ^{2}/(2\sigma_{\max}^{2}))dy\\
& =\frac{2^{(5+d)/2}}{\sigma_{\min}^{d+2}\Gamma(d/2+1)}R_{N}^{d+1}\\
& +\frac{3\sigma_{\max}}{\sigma_{\min}^{d+2}2^{d/2-1}\Gamma(d/2)}\int_{
0}^\infty\left(  \sigma_{\max}\sqrt{2t}+2R_{N}\right)  ^{d}t^{-1/2}\exp(-t)dt\\
& \equiv\text{Term1}_{N}+\text{Term2}_{N},
\end{align*}
where some standard estimates show that Term2$_{N}\lesssim_{d}R_{N}^{d},$ and
so is asymptotically dominated by Term1$_{N}$. Then  similar calculations
show that 
\begin{equation}
\left\Vert \mathcal{L}_{\mathcal{D}_{N}}\right\Vert _{L^{\infty}\left(
\mathcal{D}_{N}\right)  }+\frac{\left\Vert \mathcal{L}_{\mathcal{D}_{N}%
}\right\Vert _{L^{1}\left(  \mathbb{R}^{d}\setminus\mathcal{D}_{N}\right)  }%
}{\lambda\left(  \mathcal{D}_{N}\right)  }\leq2\frac{1+2^{(d+3)/2}}%
{\sigma_{\min}^{d+2}\pi^{d/2}}R_{N}=:L_{\mathcal{D}_{N}}.\label{LN000}%
\end{equation}
By taking into account that $d_{\mathsf{S}}=d_{\mathsf{A}}=d,$ $\mathrm{DI}(\mathcal{D}_{d,N}\times\mathsf{A}_d%
)\lesssim (A+R_N) \sqrt{d},$ we then have
by~(\ref{LNa}), (\ref{LN0}), (\ref{LN00}), (\ref{LN000}) that%
\begin{align}
  \,\left(  L_{\mathcal{D}_{N}}\,\mathrm{DI}(\mathcal{D}_{d,N}\times\mathsf{A}_d%
)+L_{\mathcal{D}%
_{N}}\,\mathrm{diam}(\mathcal{D}_{N})\mathrm{diam}(\mathsf{A})+\Lambda\right)
/\delta_{\mathcal{D}_{N}}\label{LN1}
&  \lesssim L_{\mathcal{D}_{N}}\,\mathrm{diam}(\mathcal{D}_{N})\mathrm{diam}%
(\mathsf{A})/\delta_{\mathcal{D}_{N}}\nonumber\\
&  \lesssim8A\frac{1+2^{(d+3)/2}}{\sigma_{\min}^{d+2}\pi^{d/2}}\frac{R_{N}%
^{2}}{\delta_{\mathcal{D}_{N}}}\,\nonumber\\
&  \simeq2\gamma A\left(  1+2^{(2d+3)/2}\right)  
\\
& \times \left(  \sigma_{\max}%
/\sigma_{\min}\right)  ^{d}\exp(A^{2}/\sigma_{\min}^{2})\log N\cdot N^{\gamma
}\nonumber
\end{align}
 for $N\rightarrow\infty$.
Further we have%
\begin{equation}
\frac{\Lambda}{\delta_{\mathcal{D}_{N}}^{2}}\leq\left(  \sigma_{\max}%
/\sigma_{\min}\right)  ^{2d}\exp(2A^{2}/\sigma_{\min}^{2})N^{2\gamma
},\label{LN2}%
\end{equation}
which dominates (\ref{LN1}). That is, $$
C_{1}(H,\log N,d)= \left(
\sigma_{\max}/\sigma_{\min}\right)  ^{2d}\exp(2A^{2}/\sigma_{\min}^{2})$$ and
thus $\alpha:=2\gamma<1/2$  as required.
Next we bound
\[
\mathbb{P}_{\boldsymbol{\pi}}\Bigl(S_{h}^{(x_{0},0)}\notin\mathcal{D}%
_{N}\Bigr)\leq\sup_{a_{<h}\in\mathsf{A}^{h}}\int_{\mathbb{R}^{d}\setminus
B_{R_{N}}}p_{h}^{a_{<h}}(y)\,dy.
\]
Note that
\[
p_{h}^{a_{<h}}(y)=\frac{1}{\bigl(2\pi\bar{\sigma}_{h}^{2}\bigr)^{d/2}}%
\exp\left(  -|x_{0}-y-\bar{a}_{<h}|^{2}/\bigl(2\bar{\sigma}_{h}^{2}%
\bigr)\right)
\]
where $\bar{\sigma}_{h}^{2}=\sum_{l=1}^{h}\sigma_{l}^{2}$ and $\bar{a}%
_{<h}=\sum_{l=0}^{h-1}a_{l}.$ Suppose that $N$ is large enough such that
$\left\vert x_{0}\right\vert \leq R_{N}/4$ and $HA\leq R_{N}/4$ then
\begin{align*}
\mathbb{P}_{\boldsymbol{\pi}}\Bigl(S_{h}^{(x_{0},0)}\notin\mathcal{D}%
_{N}\Bigr) &  \leq\sup_{a_{<h}\in\mathsf{A}^{h}}\frac{1}{\bigl(2\pi\bar
{\sigma}_{h}^{2}\bigr)^{d/2}}\int_{\mathbb{R}^{d}\setminus B_{R_{N}}}%
\exp\left(  -|x_{0}-y-\bar{a}_{<h}|^{2}/\bigl(2\bar{\sigma}_{h}^{2}%
\bigr)\right)  \,dy\\
&  \leq\frac{1}{\bigl(2\pi\bigr)^{d/2}}\int_{|z|>R_{N}/(2\bar{\sigma}_{h}%
)}\exp\left(  -|z|^{2}/2\right)  \,dz\\
&  =\frac{\Gamma\left(  d/2,R_{N}^{2}/(8\bar{\sigma}_{h}^{2})\right)  }%
{\Gamma\left(  d/2\right)  }%
\end{align*}
where $\Gamma\left(  s,x\right)  $ denotes the incomplete Gamma function,
which has asymptotics $\Gamma\left(  s,x\right)  \simeq x^{s-1}e^{-x}$ for
$x\rightarrow\infty.$ By plugging in the choice for $R_{N}$ we get for
$N\rightarrow\infty,$%
\begin{align*}
\mathbb{P}_{\boldsymbol{\pi}}\Bigl(S_{h}^{(x_{0},0)}   \notin\mathcal{D}%
_{N}\Bigr)&\simeq\frac{8}{2^{3d/2}\Gamma(d/2)}(R_{N}/\bar{\sigma}_{h}%
)^{d-2}\exp\left(  -R_{N}^{2}/(8\bar{\sigma}_{h}^{2})\right)  \\
&  =\frac{32}{2^{5d/2}\Gamma(d/2)}(\frac{\sigma_{\min}}{\bar{\sigma}_{h}%
})^{\left(  d-2\right)  }\gamma^{\frac{d}{2}-1}N^{-\gamma\sigma_{\min}%
^{2}/\left(  32\bar{\sigma}_{h}^{2}\right)  }\log^{\frac{d}{2}-1}N\\
&  \leq\frac{32}{2^{5d/2}\Gamma(d/2)}\left(  \sqrt{H}\frac{\sigma_{\max}%
}{\sigma_{\min}}\right)  ^{\left(  2-d\right)  _{+}}\gamma^{\frac{d}{2}%
-1}N^{-\gamma\sigma_{\min}^{2}/\left(  32\bar{\sigma}_{H}^{2}\right)  }%
\log^{\frac{d}{2}-1}N
\end{align*}
with $(2-d)_{+}:=\max(2-d,0),$ uniform in $h\in]H]$. We so may take
$\beta=\gamma\sigma_{\min}^{2}/(32\bar{\sigma}_{H}^{2})$ and%
\[
C_{2}(H,\log N,d)=\frac{32H}{2^{5d/2}\Gamma(d/2)}\left(  \sqrt{H}\frac
{\sigma_{\max}}{\sigma_{\min}}\right)  ^{\left(  2-d\right)  _{+}}%
\gamma^{\frac{d}{2}-1}\log^{\frac{d}{2}-1}N.
\]
Thus, the conditions of Corollary~\ref{addcor} are satisfied with $\alpha=2\gamma$, $ \beta=\gamma\sigma_{\min}^{2}/(32\bar{\sigma}_{H}^{2}),$  $\eta=1/2$, and $\lambda=2\log
(\sigma_{\max}/\sigma_{\min})$, where $\gamma\in (0,1/4)$ can be further chosen to ensure that $\beta\leq 1/2-\alpha$, leading to a complexity bound  $\epsilon^{-2/\beta}\log^{d/\beta}(1/\epsilon)\times \text{polynomial}(H,d)$.

\section{Linear-quadratic Gaussian (LQG) control problems}
\label{sec:lqg}
Let us consider a classical stochastic linear-quadratic-Gaussian (LQG) control
problem for controlled $d$-dimensional diffusion process of the form
\begin{equation}
dX_{t}=2\sqrt{\lambda}\,m_{t}\,dt+\sqrt{2}\,dW_{t}\label{eq:sde}%
\end{equation}
with $t\in\lbrack0,T]$, $X_{0}=x_{0}\in\mathbb{R}^{d}$, and with the objective
functional
\[
J_{0}^{m}(x_{0})=\mathbb{E}_{m,x_{0}}\big[-\int_{0}^{T}\Vert m_{t}\Vert
^{2}\,dt+F(X_{T})\big].
\]
Here $(m_{t})_{t\in\lbrack0,T]}$ with $m_{t}\in\mathbb{R}^{d}$ is the adapted control
process and $F$
is the terminal reward if $F\geq0$ or terminal costs if $F<0$.
Further,
$\lambda$ is a positive constant representing the
\textquotedblleft strength\textquotedblright\ of the control, and
$(W_{t})_{t\in\lbrack0,T]}$ is a standard Brownian motion in $\mathbb{R}^{d}$.
Our goal is to maximize the functional $J_{0}^{m}(x_{0})$ over a class of
control processes $(m_{t})_{0\leq t\leq T}$. The HJB equation for the
problem at a generic time $t\in\lbrack0,T]$ and $x\in\mathbb{R}^{d},$ that is%
\[
J_{t}^{\star}(x):=\sup_{m}\mathbb{E}_{m,x}\big[-\int_{t}^{T}\Vert m_{s}%
\Vert^{2}\,ds+F(X_{T})\big]=\sup_{m}J_{t}^{m}(x),
\]
is given by
\begin{align}
\frac{\partial}{\partial t}J_{t}^{\star}(x)+\Delta J_{t}^{\star}%
(x)+\lambda\Vert\nabla J_{t}^{\star}(x)\Vert^{2}  & =0,\label{eq:PDE_HJB}\\
J_{T}^{\star}(x)  & =F(x)\nonumber
\end{align}
(see e.g., Yong \& Zhou~\cite[Chapter 4]{yong1999stochastic}) where
$J_{t}^{\star}(x)$ of \eqref{eq:PDE_HJB} at $t=0$ is the \textquotedblleft
optimal negative cost\textquotedblright\ when the state starts from $x$.
Using the Cole-Hopf transformation $x\to \exp(\lambda J_{t}^{\star}(x))$   one transforms the nonlinear PDE \eqref{eq:PDE_HJB} to the backward heat equation.  As a result, the  solution of
\eqref{eq:PDE_HJB} admits the explicit formula
\begin{equation}
J_{t}^{\star}(x)=\frac{1}{\lambda}\log\!\bigg(\mathbb{E}\Big[\exp\!\Big(\lambda
F(x+\sqrt{2}W_{T-t})\Big)\Big]\bigg).\label{eq:HJB_formula}%
\end{equation}
This can be used to test the accuracy of the proposed algorithm.

In our implementation, we first discretize the equation \eqref{eq:sde} using the 
Euler scheme with time step $\Delta$,
\[
S_{h+1}=S_{h}+2\sqrt{\lambda}m_{h\Delta}\Delta+\sqrt{\Delta}\,\varepsilon
_{h+1},\quad h\in\lbrack H[
\]
with $H=[T/\Delta],$ $\varepsilon_{h+1}\sim\mathcal{N}(0,I_{d})$ and
$S_{0}=x_{0}$. We then consider the discrete time controlled Markov chain%
\begin{equation}
S_{h+1}=S_{h}+a_{h}+\sqrt{\Delta}\,\varepsilon_{h+1},\quad h\in\lbrack
H[,\label{eq:discr-lqg}%
\end{equation}
by taking as control at time $h,$
\[
a_{h}:=2\sqrt{\lambda}m_{h\Delta}\Delta\in\lbrack-A,A]^{d}\text{ \ \ for some
}A>0.
\]
As such the conditional density of the Markov chain \eqref{eq:discr-lqg} is
Gaussian and of the form \eqref{eq:gauss-trans-dens} with $\sigma_{h}%
^{2}=\Delta$ for every $h.$ Thus the objective is to maximize the functional%
\[
V_{0}^{\boldsymbol{\pi}}(x_{0})=\mathbb{E}_{\boldsymbol{\pi},x_{0}}\left[
-\frac{1}{4\lambda\Delta}\sum_{k=0}^{H-1}\Vert\pi_{k}(S_{k})\Vert^{2}%
+F(S_{H})\right]
\]
over all policies $\boldsymbol{\pi}=(\pi_{k}(S_{k}))_{k\in\lbrack H[},$ where
$\pi_{k}:$ $\mathbb{R}^{d}\rightarrow\lbrack-A,A]^{d}.$ The optimal value of
the objective as seen from a generic time $h$ with starting point $S_{h}%
=x\in\mathbb{R}^{d}$ is given by
\[
V_{h}^{\star}(x)=\sup_{\pi_{h},\ldots\pi_{H-1}}\mathbb{E}_{\boldsymbol{\pi}%
,x}\left[  \left.  -\frac{1}{4\lambda\Delta}\sum_{k=h}^{H-1}\Vert\pi_{k}%
(S_{k})\Vert^{2}+F(S_{H})\right\vert S_{h}=x\right]  ,
\]
and satisfies the backward dynamic program
\[
V_{h}^{\star}(x)=\max_{a\in\lbrack-A,A]^{d}}\Bigl(-\frac{\Vert a\Vert^{2}%
}{4\lambda\Delta}+\mathbb{E}\left[  V_{h+1}^{\star}(x+a+\sqrt{\Delta
}\,\varepsilon_{h+1})\right]  \Bigr),\quad h=H-1,\ldots,0,
\]
with $V_{H}^{\star}(x)=F(x).$ 
\par
In our numerical experiments, we take 
\[
F(x)=\pm\log((1+\Vert x\Vert^{2})/2),%
\]
$T=0.2$ and $\Delta=0.01,$ hence   $H=20.$ Actions are
sampled uniformly on $[-1,1]^d$ and the optimization is performed over the resulting grid. The representative controls $b_0,\ldots,b_{H-1}$ are all taken to be zero. The results for
dimension $d=1$ are presented in Table~\ref{tab:d11} and Table~\ref{tab:d12}. They are obtained using a grid  of $50$  actions.  The value of  the explicit formula \eqref{eq:HJB_formula} is approximated using MC with $10000$ samples. 
\begin{table}
    \caption{Results for  $F(x)=- \log ((1+\|x\|^2)/2)$ and $d=1.$
            The explicit formula \eqref{eq:HJB_formula} gives $0.4542$. \label{tab:d11}}

            \begin{tabular}{ |p{3cm}|p{3cm}|p{4cm}|p{3cm}|}
            \hline
            mean & bias & standard deviation & number of trajectories\\
            \hline
            \hline
            0.464 & 0.034 & 0.044 & 10 \\
            0.459 & 0.007 & 0.009 & 100\\
            0.451 &  0.007 & 0.009 & 200 \\
            0.451 & 0.004 & 0.004 & 500 \\
            \hline 
            \end{tabular} 
\end{table}
\begin{table}
    \caption{Results for  $F(x)=\log ((1+\|x\|^2)/2)$ and $d=1.$
            The explicit formula \eqref{eq:HJB_formula} gives $-0.357$. \label{tab:d12}} 
            \begin{tabular}{ |p{3cm}|p{3cm}|p{4cm}|p{3cm}|}
            \hline
            mean & bias & standard deviation & number of trajectories\\
            \hline
            \hline
            -0.416 & 0.096 & 0.077 & 10 \\
            -0.377 & 0.021 & 0.016 & 100\\
            -0.373 &  0.014 & 0.011 & 200 \\
            -0.369 & 0.008 & 0.006 & 500 \\
            \hline 
            \end{tabular} 
\end{table}
The results for
dimension $d=5$ are presented in Table~\ref{tab:d51} and Table~\ref{tab:d52}, and obtained using a grid  of $400$  actions. 
\begin{table}
    \caption{Results for  $F(x)=-\log ((1+\|x\|^2)/2)$ and $d=5.$
            The explicit formula \eqref{eq:HJB_formula} gives $-0.2474$. \label{tab:d51}} 
             \begin{tabular}{ |p{3cm}|p{3cm}|p{4cm}|p{3cm}|}
                \hline
                mean & bias & standard deviation & number of trajectories\\
                \hline
                \hline
                -0.19 & 0.12 & 0.15 & 10 \\
                -0.22 & 0.026 & 0.034 & 100\\
                -0.21 &  0.011 & 0.015 & 200 \\
                -0.23 & 0.011 & 0.013 & 500 \\
                \hline 
            \end{tabular} 
\end{table}
\begin{table}
    \caption{Results for  $F(x)=\log ((1+\|x\|^2)/2)$ and $d=5.$ The explicit formula \eqref{eq:HJB_formula} gives $0.4054$. \label{tab:d52}}
            \begin{tabular}{ |p{3cm}|p{3cm}|p{4cm}|p{3cm}|}
                \hline
                mean & bias & standard deviation & number of trajectories\\
                \hline
                \hline
                0.378 & 0.13 & 0.17 & 10 \\
                0.343 & 0.029 & 0.038 & 100\\
                0.337 &  0.014 & 0.018 & 200 \\
                0.337 & 0.013 & 0.016 & 500 \\
            \hline 
            \end{tabular} 
\end{table}
Let us further study the performance of our algorithm for different values of the parameter $\lambda.$ Figure~\ref{fig:lambda} shows the estimates of $V^\star_0(0)$ (red line) together with the values obtained from \eqref{eq:HJB_formula} using MC with $10000$ paths (green line) for the case $F(x)=\log ((1+\|x\|^2)/2)$ and $d=1.$
\begin{figure}
    \centering
\includegraphics[width=0.7\textwidth]{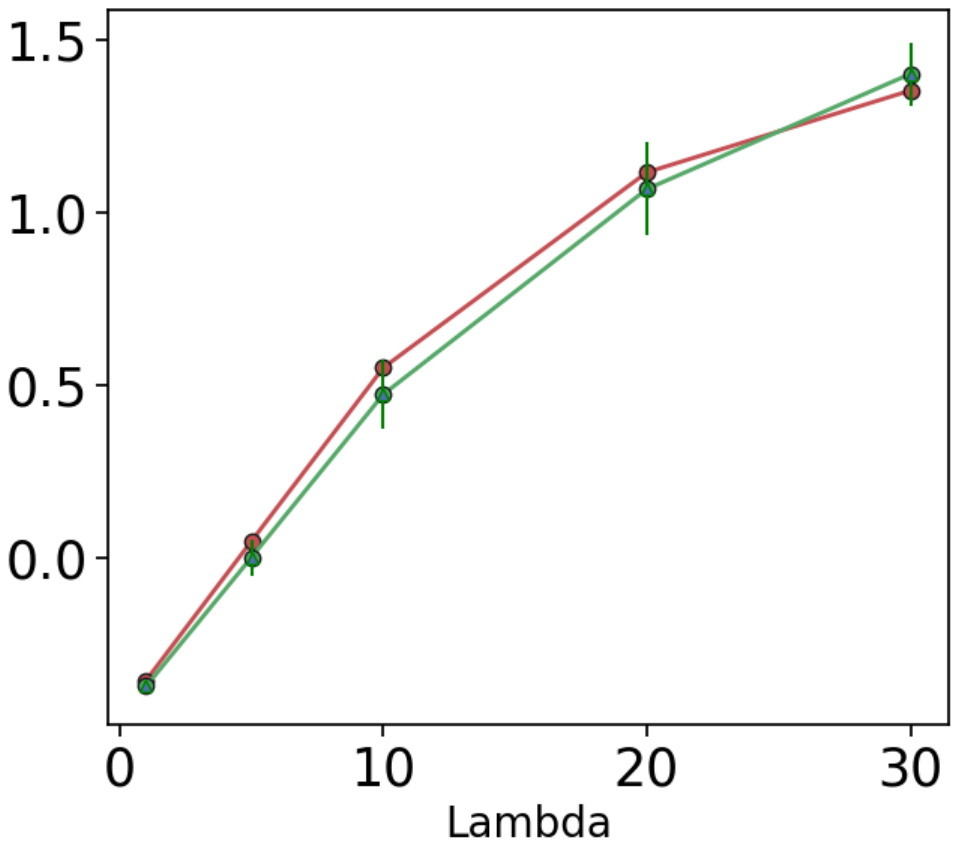}
    \caption{Results for different values of the parameter $\lambda$  and $F(x)=\log((1+|x|)^2/2)$.\label{fig:lambda}}
   
\end{figure}

\section{Proofs}
\subsection{Proof of Proposition~\ref{prop:noncomp-compl}}

Let us start with a simple observation. If $x=x(t)$ satisfies the equation%
\begin{equation}
\frac{x}{\log^{b}x}=t,\text{ \ \ }t>0,\text{ }b>0,\text{\ \ }\label{aseq}%
\end{equation}
one then has for $t\uparrow\infty,$%
\begin{equation}
x(t)=(1+o(1))t\log^{b}t.\label{aseqsol}%
\end{equation}
This is easily seen as follows. Clearly $x\uparrow\infty$ when $t\uparrow
\infty.$ Then, by setting $z=\log x$ and $s=\log t$ we get $e^{z}/z^{b}%
=e^{s},$ and we so may write%
\[
z\left(  1-b\frac{\log z}{z}\right)  =s
\]
where $z\uparrow\infty$ when $s\uparrow\infty.$ Since $z^{-1}\log z=o(1)$ for
$s\uparrow\infty$ we conclude that $z=(1+o(1))s,$ for $s\uparrow\infty,$ and
hence%
\begin{equation}
\log x=(1+o(1))\log t,\text{ \ \ }t\uparrow\infty.\label{logsol}%
\end{equation}
Next, substituting (\ref{logsol}) in (\ref{aseq}) yields (\ref{aseqsol}). 

Now suppose we
need to achieve an accuracy $\epsilon>0.$ 
The
assumptions imply that it is sufficient to choose $N$
large enough such that%
\begin{gather*}
GcH^{2+q}N^{\alpha-\frac{1}{2}}\log^{q}N\leq\frac{\epsilon}{2}\text{
\ \ and \ \ }GcH^{1+q}N^{-\beta}\log^{q}N\leq\frac{\epsilon}{2},\text{
\ \ or}\\
\frac{N}{\log^{2q/(1-2\alpha)}N}\geq\left(  \frac{2GcH^{2+q}}{\epsilon
}\right)  ^{2/(1-2\alpha)}\text{ and\ }\frac{N}{\log^{q/\beta}N}\geq\left(
\frac{2GcH^{1+q}}{\epsilon}\right)  ^{1/\beta},
\end{gather*}
where $c\equiv c_d$ and $q\equiv q_d$.
By considering equalities instead of inequalities we get equations of the form
(\ref{aseq}) with asymptotic solutions due to (\ref{logsol}),%
\begin{align*}
N &  =\left(  1+o(1)\right)  \left(  \frac{2^{q+1}GcH^{2+q}}{\left(
1-2\alpha\right)  ^{q}}\right)  ^{2/(1-2\alpha)}\frac{\log^{2q/(1-2\alpha
)}\frac{1}{\epsilon}}{\epsilon^{2/(1-2\alpha)}}\qquad\text{and}\\
N &  =(1+o(1))\left(  \frac{2GcH^{1+q}}{\beta^{q}}\right)  ^{1/\beta}%
\frac{\log^{q/\beta}\frac{1}{\epsilon}}{\epsilon^{1/\beta}},\text{ \ \ for
\ \ }\epsilon\downarrow0,
\end{align*}
respectively.
We thus end up with a complexity $\mathcal{C}\left(
\epsilon\right)  =HN^2$ which is bounded by (\ref{noncompcomp0}).

\newpage
\appendix

\section{Some auxiliary notions}\label{sec:emp}
The Orlicz 2-norm of a real valued random variable $\eta$ with respect to the
function $\psi_2(x)=e^{x^{2}}-1$, $x\in\mathbb{R}$, is defined by $\Vert
\eta\Vert_{\psi_2}:=\inf\{t>0:\mathbb{E}\left[  {\exp}\left(  {\eta
^{2}/t^{2}}\right)  \right]  \leq2\}$. We say that $\eta$ is
\emph{sub-Gaussian} if $\Vert\eta\Vert_{\psi_2}<\infty$. In particular,
this implies that for some constants $C,c>0$,
\[
\mathbb{P}(|\eta|\geq t)\leq2\exp\left(  -\frac{ct^{2}}{\Vert\eta
\Vert_{\psi_2}^{2}}\right)  \text{\ \ and \ \ }\mathbb{E}[|\eta|^{p}%
]^{1/p}\leq C\sqrt{p}\Vert\eta\Vert_{\psi_2}\text{ \ \ for all \ \ }p\geq1.
\]
Consider a real valued random process $(X_{t})_{t\in\mathcal{T}}$ on a metric
parameter space $(\mathcal{T},\mathsf{d})$. We say that the process has
\emph{sub-Gaussian increments} if there exists $K\geq0$ such that
\[
\Vert X_{t}-X_{s}\Vert_{\psi_2}\leq K\mathsf{d}(t,s),\quad\forall
t,s\in\mathcal{T}.
\]
Let $(\mathsf{Y},\rho)$ be a metric space and $\mathsf{X\subseteq Y}$. For
$\varepsilon>0$, we denote by $\mathcal{N}(\mathsf{X},\rho,\varepsilon)$ the
covering number of the set $\mathsf{X}$ with respect to the metric $\rho$,
that is, the smallest cardinality of a set (or net) of $\varepsilon$-balls in
the metric $\rho$ that covers $\mathsf{X}$. Then $\log\mathcal{N}%
(\mathsf{X},\rho,\varepsilon)$ is called the metric entropy (or Dudley integral) of $\mathsf{X}$
and
\[
\mathrm{DI}(\mathsf{X}):=\int_{0}^{\operatorname{diam}(\mathsf{X})}\sqrt
{\log\mathcal{N}\bigl(\mathsf{X},\rho,u\bigr)}\,du
\]
with $\operatorname{diam}(\mathsf{X}):=\max
_{x,x^{\prime}\in\mathsf{X}}\rho(x,x^{\prime}),$ is called the Dudley
integral. For example, if $|\mathsf{X}|<\infty$ and $\rho(x,x^{\prime
})=1_{\{x\neq x^{\prime}\}}$ we get $\mathrm{DI}(\mathsf{X})=\sqrt
{\log|\mathsf{X}|}.$

\section{Estimation of mean uniformly in parameter}

The following proposition holds.
\begin{proposition}
\label{prop:unif-exp}
Let \(f\) be a function on \(\mathsf{X}\times \Xi\) such that
\begin{eqnarray}
\left|f(x,\xi)-f(x',\xi)\right|\leq L\rho(x,x')
\end{eqnarray}
with some constant \(L>0.\) Furthermore assume that \(\|f\|_{\infty}\leq F<\infty\) for some \(F>0.\) Let \(\xi_n,\) \(n=1,\ldots,N,\) be i.i.d. sample from a distribution on \( \Xi.\) Then we have
\begin{eqnarray*}
\mathbb{E}^{1/p}\left[\sup_{x\in \mathsf{X} }\left\vert \frac{1}{N}\sum_{n=1}^{N}\left(f(x,\xi_n)-\mathbb{E}f(x,\xi_n)\right)\right\vert ^{p}\right]\lesssim \frac{L\, \mathrm{DI}(\mathsf{X}) +   (L\, \mathrm{diam}(\mathsf{X})  +  F) \sqrt{p}}{\sqrt{N}},
\end{eqnarray*}
where $\lesssim$ may be interpreted as $\le$ up to a
natural constant.
\end{proposition}
\begin{proof}
Denote
\[
Z(x):=\frac{1}{\sqrt{N}}\sum_{n=1}^{N}\left(f(x,\xi_n)-M_{f}(x)\right)
\]
with $M_{f}(x)=\mathbb{E}[f(x,\xi)],$ that is, $Z(x)$ is
a centered random process on the metric space $(\mathsf{X},\rho)$.
Below we show that the process $Z(x)$ has sub-Gaussian increments.
In order to show it, let us introduce
\[
Z_{n}=f(x,\xi_n)-M_{f}(x)-f(x',\xi_n)+M_{f}(x').
\]
Under our assumptions we get
\[
\|Z_{n}\|_{\psi_{2}}\lesssim L\rho(x,x'),
\]
that is, $Z_{n}$ is sub-Gaussian for any $n=1,\ldots,N.$ Since
\[
\ensuremath{Z(x)-Z(x')=N^{-1/2}\sum_{n=1}^{N}Z_{n}},
\]
is a sum of independent sub-Gaussian r.v, we may apply \cite[Proposition 2.6.1 and Eq. (2.16)]{Vershynin}) to obtain that
$Z(x)$ has sub-Gaussian increments with parameter $K\asymp L$. Fix
some $x_0\in \mathsf{X}.$ By the triangular inequality,
\[
\sup_{x\in\mathsf{X}}|Z(x)|\le\sup_{x,x'\in\mathsf{X}}|Z(x)-Z(x')|+\left|Z(x_0)\right|.
\]
By the Dudley integral inequality, e.g. \cite[Theorem
8.1.6]{Vershynin}, for any $\delta\in(0,1)$,
\[
\sup_{x,x'\in\mathsf{X}}|Z(x)-Z(x')|\lesssim L\bigl[\mathrm{DI}(\mathsf{X})+\mathrm{diam}(\mathsf{X})\sqrt{\log(2/\delta)}\bigr]
\]
holds with probability at least $1-\delta$. Again, under our assumptions,
$Z(x_0)$ is a sum of i.i.d. bounded centered random variables with
$\psi_{2}$-norm bounded by $F$. Hence, applying Hoeffding's inequality,
e.g. \cite[Theorem 2.6.2.]{Vershynin}, for any
$\delta\in(0,1)$,
\[
|Z(x_0)|\lesssim F\sqrt{\log(1/\delta)}.
\]
\end{proof}

\section{Heuristic motivation for approximation (\ref{appr})}\label{heurw}
Let us give a heuristic motivation for the choice of the weights in approximation  (\ref{appr}). First
consider that for large $N$,%
\begin{equation}
\frac{1}{N-1}\sum_{k^{\prime}=1,\,k^{\prime}\neq n^{\prime}}^{N}p_{h+1}%
^{b_{h}}(S_{h+1}^{(n^{\prime)}}|S_{h}^{(k^{\prime})})\approx\int
p_{h+1}^{b_{h}}(S_{h+1}^{(n^{\prime)}}|z)p_{h}^{_{b_{<h}}}(z)dz=p_{h+1}%
^{_{b_{<h+1}}}(S_{h+1}^{(n^{\prime)}}),\label{apprheu}%
\end{equation}
and so for large $N$ the denominator of the weight $w_{h,n,N}$ in (\ref{appr}) becomes%
\begin{align*}
&  \sum_{n^{\prime}=1}^{N}\frac{p_{h+1}^{a}(S_{h+1}^{(n^{\prime})}|x)}%
{\sum_{k^{\prime}=1,\,k^{\prime}\neq n^{\prime}}^{N}p_{h+1}^{b_{h}}%
(S_{h+1}^{(n^{\prime)}}|S_{h}^{(k^{\prime})})}\approx\frac{1}{N-1}%
\sum_{n^{\prime}=1}^{N}\frac{p_{h+1}^{a}(S_{h+1}^{(n^{\prime})}|x)}%
{p_{h+1}^{_{b_{<h+1}}}(S_{h+1}^{(n^{\prime)}})}\\
&  \approx\frac{N}{N-1}\int\frac{p_{h+1}^{a}(z^{\prime}|x)}{p_{h+1}%
^{_{b_{<h+1}}}(z^{\prime})}p_{h+1}^{_{b_{<h+1}}}(z^{\prime})dz^{\prime}%
=\frac{N}{N-1}.%
\end{align*}
This in turn means that for any bounded $f$ on $\mathsf{S}$ and large $N,$ by
using (\ref{apprheu}) once again,%
\begin{align*}
&  \sum_{n=1}^{N}f_{h+1}(S_{h+1}^{(n)})w_{h,n,N}(x,a)\\
&  \approx\frac{1}{N}\sum_{n=1}^{N}f_{h+1}(S_{h+1}^{(n)})\frac{p_{h+1}%
^{a}(S_{h+1}^{(n)}|x)}{\frac{1}{N-1}\sum_{k=1,\,k\neq n}^{N}p_{h+1}^{b_{h}%
}(S_{h+1}^{(n)}|S_{h}^{(k)})}\\
&  \approx\frac{1}{N}\sum_{n=1}^{N}\frac{f_{h+1}(S_{h+1}^{(n)})p_{h+1}%
^{a}(S_{h+1}^{(n)}|x)}{p_{h+1}^{_{b_{<h+1}}}(S_{h+1}^{(n^{)}})}\approx
\int\frac{f_{h+1}(z)p_{h+1}^{a}(z|x)}{p_{h+1}^{_{b_{<h+1}}}(z)}p_{h+1}%
^{_{b_{<h+1}}}(z)dz\\
&  =\int f_{h+1}(z)p_{h+1}^{a}(z|x)dz=\mathbb{E}_{S_{h+1}\sim P_{h+1}%
(\cdot|x,a)}\left[  f_{h+1}(S_{h+1})\right]  .
\end{align*}

\section{Reflection of Markov chains}
\label{sec:refl}

Let $\left(  S_{h}\right)  _{h=0,\ldots,H}$ be a Markov chain in
$\mathbb{R}^{d}$ with one-step transition density $p_{h+1}(y|x)$ for $0\leq
h<H.$ Let further $\mathcal{D\subset}$ $\mathbb{R}^{d}$ be a compact Borel
subset and $q(dy)=\lambda(dy)/\lambda(\mathcal{D})$ with $\lambda$ being
Lebesgue measure on $\mathbb{R}^{d}$. We then construct a Markov chain
$\left(  S_{h}^{\mathcal{D}}\right)  _{h=0,\ldots,H}$ in $\mathcal{D}$ as
follows: Suppose $S_{h}^{\mathcal{D}}=x\in\mathcal{D}$. Let $Y\in
\mathbb{R}^{d}$ be a random variable with density $p_{h+1}(y|x)$ and
$Q\in\mathcal{D}$ be a random variable independent of $Y$ with density
$\lambda^{-1}(\mathcal{D}),$ hence $Q$ is uniformly distributed on
$\mathcal{D}$. We then define%
\[
S_{h+1}^{\mathcal{D}}:=\left\{
\begin{tabular}
[c]{l}%
$Y$ \ \ if \ \ $Y\in\mathcal{D}$\\
$Q$ \ \ if \ \ $Y\notin\mathcal{D}$%
\end{tabular}
\right.  .
\]
For any non-negative Borel function $f$ in $\mathbb{R}^{d}$ with support
$\mathcal{D}$ one thus has%
\begin{align*}
\mathbb{E}\left[  f(S_{h+1}^{\mathcal{D}})\right]    & =\mathbb{E}\left[
f(Y)1_{\left\{  Y\in\mathcal{D}\right\}  }\right]  +\mathbb{E}\left[
f(Q)1_{\left\{  Y\notin\mathcal{D}\right\}  }\right]  \\
& =\int_{\mathcal{D}}f(y)p_{h+1}(y|x)dy+\frac{1}{\lambda(\mathcal{D})}%
\int_{\mathcal{D}}f(y)dy\int_{\mathbb{R}^{d}\setminus\mathcal{D}}%
p_{h+1}(z|x)dz.
\end{align*}
Hence, $S_{h+1}^{\mathcal{D}}$ is governed by the one-step transition density
$p_{h+1}^{\mathcal{D}}$ on $\mathcal{D}\times\mathcal{D}$ given by%
\[
p_{h+1}^{\mathcal{D}}(y|x)=p_{h+1}(y|x)+\frac{1}{\lambda(\mathcal{D})}%
\int_{\mathbb{R}^{d}\setminus\mathcal{D}}p_{h+1}(z|x)dz,\text{ \ \ }%
x,y\in\mathcal{D}.
\]
Consider furthermore the stopping time%
\[
\tau^{\mathcal{D}}:=\min\left\{  h:S_{h}\notin\mathcal{D}\right\}  \text{
\ \ with \ \ }S_{0}=x_{0}\in\mathcal{D}\text{.}%
\]
It is not difficult to see that one has that
\begin{equation}
\left(  S_{h}:0\leq h<\tau^{\mathcal{D}}\right)  \overset{\mathcal{L}}%
{=}\left(  S_{h}^{\mathcal{D}}:0\leq h<\tau^{\mathcal{D}}\right)
,\label{idlaw}%
\end{equation}
Loosely speaking, $S_{h}^{\mathcal{D}}$ behaves like $S_{h}$ before the first
exit of the set $\mathcal{D}$. We will say that $S_{h}^{\mathcal{D}}$ is the
\textit{reflected Markov chain} obtained from reflecting the process $S_{h}$
in $\mathcal{D}$.

\section*{Declarations}

\noindent\textbf{Conflict of interest. } The authors have no competing interests to declare that are relevant to the content of this article.
\\
\noindent\textbf{Data availability. } Data sharing is not applicable to this article as no datasets were generated or analyzed during the current study.

\subsection*{Acknowledgments}
D.B. and V.Z. gratefully acknowledge financial support from the German science foundation (DFG), grant Nr.497300407. 
J.S. gratefully acknowledges financial support from the German science foundation (DFG) via the
cluster of excellence MATH+, project AA4-2.

\bibliography{biblio-rl,stop}
\bibliographystyle{plain}

\end{document}